\documentclass[11pt]{article}

\usepackage{sudoku}
\usepackage{amssymb,amsthm,amsmath}
\usepackage[curve]{xypic}

\newtheorem{theorem}{Theorem}[]
\newtheorem{lemma}[theorem]{Lemma}
\newtheorem{definition}[theorem]{Definition}

\theoremstyle{remark}

\numberwithin{equation}{section}
\theoremstyle{theorem}

\usepackage[english]{babel}
\usepackage[pdftex]{graphicx}
\usepackage{color}

\usepackage{epsfig}

\begin{document}

\title{Nest graphs and minimal complete symmetry groups for magic Sudoku variants}

\author{E. Arnold, R. Field, J. Lorch, S. Lucas, and L. Taalman}

\date{\today}

\maketitle

\section{Introduction}
\label{S:intro}

Felgenhauer and Jarvis famously showed in \cite{Felgenhauer}, although it was first mentioned earlier, in \cite{QSCGZ}, that there are 6,670,903,752,021,072,936,960 possible completed Sudoku boards. In a later paper, Jarvis and Russell \cite{Russell} used a Sudoku symmetry group of size $3,359,232\cdot 9!=1,218,998,108,160$ and Burnside's Lemma to show that there are 5,472,730,538 essentially different Sudoku boards.  Both of these results required extensive use of computers as magnitude of the numbers makes non-computer exploration of these problems prohibitively difficult.  The ongoing goal of this project is to find and implement methods to attack these and similar questions without the aid of a computer.

One step in this direction is to reduce the size of the symmetry group with purely algebraic, non-computer methods.   The strategy of \cite{Arnold2}, applied to the analogous symmetry group for a $4 \times 4$ Sudoku variation known as Shidoku, was to  partition the set of Shidoku boards into so-called $H_4$-nests and $S_4$-nests and then use the interplay between the physical and relabeling symmetries to find certain subgroups of $G_4$ that were both complete and minimal.  A symmetry group is {\em complete} if its action partitions the set of Shidoku boards into the two possible orbits, and {\em minimal} if no group of smaller size would do the same.

 In \cite{LorchWeld}, Lorch and Weld investigated a $9 \times 9$ variation of Sudoku called {\em modular-magic Sudoku} that has sufficiently restrictive internal structure to allow for non-computer investigation.  In this paper we will apply the techniques from \cite{Arnold2} to find a minimal complete symmetry group for the modular-magic Sudoku variation studied in \cite{LorchWeld}, as well as for another Sudoku variation that we will call {\em semi-magic Sudoku}.
 
We conclude this paper with  a simple calculator computation which leads to the non-obvious fact that the full Sudoku symmetry group is, in fact,  already minimal and complete.

\section{Modular-magic Sudoku}
\label{S:mm_props}

A Sudoku board is a $9 \time 9$ grid with nine $3 \times 3$ designated blocks.  We call the rows,  columns and diagonals of these blocks {\em mini-rows}, {\em mini-columns} and {\em mini-diagonals} respectively.  We call a rows and columns of $3 \times 3$ blocks  {\em bands} and {\em pillars} respectively. A {\em modular-magic Sudoku board} is a standard Sudoku board using the numbers 0--8 with the additional constraint that each $3 \times 3$ block is a magic square modulo 9, in the sense that the entries of every mini-row, mini-column and mini-diagonal have a sum that is divisible by $9$; see Figure~\ref{F:722}.   In this section we find a complete minimal symmetry group for modular-magic Sudoku (Theorem \ref{T:maintheorem1}). 

\begin{figure}[h]
\begin{center} 
\begin{sudoku-block}
|0|2|7|3|1|5|6|4|8|.
|1|3|5|8|6|4|2|0|7|.
|8|4|6|7|2|0|1|5|3|.
|3|5|1|6|4|8|0|7|2|.
|4|6|8|2|0|7|5|3|1|.
|2|7|0|1|5|3|4|8|6|.
|6|8|4|0|7|2|3|1|5|.
|7|0|2|5|3|1|8|6|4|.
|5|1|3|4|8|6|7|2|0|.
\end{sudoku-block}
\vspace{-1\baselineskip}
\end{center}
\caption{A modular-magic Sudoku board.}
\label{F:722}
\end{figure}

\subsection{Modular-magic Sudoku Properties}
\label{SS:mm_sym}

In this subsection, we review some facts about modular-magic Sudoku boards. For details see \cite{LorchWeld}.

Most, but not all, of the usual physical Sudoku symmetries in \cite{Felgenhauer} are valid for modular-magic Sudoku.  In particular, band swaps, pillar swaps, transpose, rotation, and  row or column swaps that  do not change the set of entries in the mini-diagonals, all preserve the modular-magic condition.  However, row or column swaps that change the center cell of a block are not modular-magic Sudoku symmetries. For example, swapping the first and second rows of the board in Figure~\ref{F:722} would result in a board that fails the modular-magic mini-diagional condition.  The order of the full group $H_{mm}$ of physical modular-magic Sudoku symmetries is $4608$.  

The set of allowable relabeling symmetries is greatly reduced for modular-magic Sudoku, as very few relabelings will preserve the modular-magic condition.  In fact, there are only $36$ elements in the group $S_{mm}$ of modular-magic relabeling symmetries on the digits 0--8, namely, the permutation
$$\rho=(12)(45)(78)$$

\noindent
and permutations of the form
$$\mu_{k,l}(n)=kn+l \mod 9$$

\noindent
for $k\in\{1,2,4,5,7,8\}$ and $l\in\{0,3,6\}$.  
Together with the physical symmetries this gives a full modular-magic Sudoku symmetry group $G_{mm}$ of size 165,888.  Since there are only 32,256 possible modular-magic Sudoku boards, this symmetry group is clearly larger than necessary.  Furthermore, the largest orbit of $G_{mm}$ has 27,648 elements, hence this is the smallest size possible for a complete modular-magic Sudoku symmetry group.  Our goal is to determine if this minimum can be obtained.

In \cite{LorchWeld} it is shown that the set of modular-magic boards breaks into two orbits under the action of $G_{mm}$, with representatives shown in Figure \ref{F:representatives}.

\begin{figure}[h]
\begin{center}
\begin{sudoku-block}
      |1 | 8 | 0 |  7 | 5 |6 | 4 | 2 | 3 |.
      |2 | 3 | 4 |  8 | 0 |1 | 5 | 6 | 7  |.
      |6 | 7 | 5 |  3 | 4 |2 | 0 | 1 | 8 |.
      |7 | 5 | 6 |  4 | 2 |3 | 1 | 8 | 0 |.
      |8 | 0 | 1 |  5 | 6 |7 | 2 | 3 | 4 |.
      |3 | 4 | 2 |  0 | 1 |8 | 6 | 7 | 5 |.
      |4 | 2 | 3 |  1 | 8 |0 | 7 | 5 | 6 |.
      |5 | 6 | 7 |  2 | 3 |4 | 8 | 0 | 1 |.
      |0 | 1 | 8 |  6 | 7 |5 | 3 | 4 | 2 |.
      \end{sudoku-block}
      \qquad  \qquad
    \begin{sudoku-block}
      |1 | 8 | 0 |  7 | 5 |6 | 4 | 2 | 3 |.
      |2 | 3 | 4 |  8 | 0 |1 | 5 | 6 | 7  |.
      |6 | 7 | 5 |  3 | 4 |2 | 0 | 1 | 8 |.
      |8 | 4 | 6 |  5 | 1 |3 | 2 | 7 | 0 |.
      |7 | 0 | 2 |  4 | 6 |8 | 1 | 3 | 5 |.
     | 3 | 5 | 1 |  0 | 2 |7 | 6 | 8 | 4 |.
      |5 | 1 | 3 |  2 | 7 |0 | 8 | 4 | 6 |.
     | 4 | 6 | 8 |  1 | 3 |5 | 7 | 0 | 2 |.
     | 0 | 2 | 7 |  6 | 8 |4 | 3 | 5 | 1 |.
    \end{sudoku-block}\ .
\end{center}
\caption{Representatives of the two $G_{mm}$-orbits in the set of modular-magic boards.}
\label{F:representatives}
\end{figure}

Every $3 \times 3$ block in a modular-magic Sudoku board has two mini-diagonals, one of which must be from the set $\{0,3,6\}$.  Therefore each modular-magic Sudoku board has exactly three blocks with center entry 0, three with center entry 3, and three with center entry 6.  In any block we will call the {\em off-diagonal set} the set of the two corner entries of the mini-diagonal whose entries are not from  $\{0,3,6\}$.  For example, in the first modular-magic Sudoku board from Figure~\ref{F:representatives}, the off-diagonal set of the first block is $\{1,5\}$.  The following lemma will be useful for proving our first theorem in the next section.

\begin{lemma}\label{l:twoequal}
If $M$ is a modular-magic Sudoku board then the three blocks with center $j$ have at least two off-diagonal sets in common, for $j=0,3,6$.
\end{lemma}

\begin{proof}
Observe that the lemma holds for the two $G_{mm}$-orbit representatives in Figure~\ref{F:representatives}, and further that the property described in the lemma is invariant under the action of $G_{mm}$. The latter assertion is quickly seen by applying generators of $G_{mm}$ to these representatives. We conclude that the lemma holds for all modular-magic sudoku boards.
\end{proof}

\subsection{$H$-nest representatives for modular-magic Sudoku}
\label{SS:mm_nests}

Following the method of \cite{Arnold2}, in this subsection we identify modular-magic Sudoku boards that can serve as representatives for equivalence classes, called $H_{mm}$-nests, defined from the modular-magic physical symmetries.  This will allow us to identify a restricted set of relabeling symmetries that, together with the physical symmetries, forms a minimal complete modular-magic Sudoku symmetry group.

We say that two modular-magic Sudoku boards are in the same $H_{mm}$-{\em nest}  when one can be obtained from the other by a sequence of physical symmetries from $H_{mm}$.  In Theorem \ref{T:Hmm_rep} we describe a unique representative for each $H_{mm}$-nest.

\begin{theorem}
Each $H_{mm}$-nest has a unique representative of the form shown in Figure~\ref{F:Hmm_nest}, where
$\alpha  <\beta $ and the two entries marked $\gamma$ are equal.
\label{T:Hmm_rep}
\end{theorem}

\begin{figure}[h]
\begin{center}
\begin{sudoku-block}
|0| |{\mbox{$\alpha$}}|3| | |6| | |.
| |3| | |6| |  |0| |.
|{\mbox{$\beta$}}| |6 | | |0| | |3|.
|3| | |6| | |0| |{\mbox{$\gamma$}} |.
| |6| |  |0| | |3||.
| | |0| | |3| | |6|.
|6| | |0| |{\mbox{$\gamma$}} |3| | |.
| |0| | |3| | |6| |.
| | |3| | |6| | |0|.
\end{sudoku-block}
\vspace{-1\baselineskip}
\end{center}
\caption{An $H_{mm}$-nest representative.}
\label{F:Hmm_nest}
\end{figure}

\begin{proof}
Band, pillar, row, and column swaps from $H_{mm}$ can transform the upper-left block of any modular-magic board into one with $\{0,3,6\}$ on the decreasing mini-diagonal as shown in Figure~\ref{F:Hmm_nest}, and with further band, pillar, row, and column swaps from $H_{mm}$ we can obtain a board $M$ of the form shown in Figure \ref{F:M}.

\begin{figure}[h]
\begin{center}
\begin{sudoku-block}
|0| |{\mbox{$\alpha _1$}}|3| | |6| | |.
| |3| | |6| |  |0| |.
|{\mbox{$\beta _1$}}| |6 | | |0| | |3|.
|3| | |6| | |0| |{\mbox{$\alpha _2$}} |.
| |6| |  |0| | |3||.
| | |0| | |3|{\mbox{$\beta _2$}} | |6|.
|6| | |0| |{\mbox{$\alpha _3$}} |3| | |.
| |0| | |3| | |6| |.
| | |3|{\mbox{$\beta _3 $}} | |6| | |0|.
\end{sudoku-block}
\end{center}
\caption{Modular-magic sudoku board $M$.}
\label{F:M}
\end{figure}

In light of Lemma \ref{l:twoequal}, we can apply band/pillar permutations to ensure that $\{\alpha _2,\beta _2\}=\{\alpha _3,\beta _3\}$.  By applying the transpose symmetry in $H_{mm}$ (if necessary) we may assume that $\alpha _1<\beta _1$.  Since $\alpha _1 + 3 + \beta _1$ must be divisible by 9, the condition $\alpha _1 <\beta _1 $ means that we must have $\alpha _1 =1$, $2$, or $7$.  By completing partial boards it can be shown that if $\alpha_1=1$, then the only possible values for $\alpha_2$ and $\alpha _3$ are $1$, $2$, and $8$. This, together with the fact that $\{\alpha _2,\beta _2\}=\{\alpha _3,\beta _3\}$, implies that $\alpha _2=\alpha _3$ when $\alpha _1= 1$. A similar argument can be applied for the other possible values of $\alpha _1$, and therefore $M$ has the form of Figure \ref{F:Hmm_nest}.

We denote boards as depicted in Figure \ref{F:Hmm_nest} by $[\alpha ,\gamma ]$. Note that this data completely determines every entry of the board. Suppose that $[\alpha ,\gamma ]$ and $[\alpha ', \gamma ']$ are $H_{mm}$-equivalent. Then either $\alpha =\alpha '$ and $\gamma =\gamma '$, in which case the boards are identical, or $\alpha =\gamma '$, $\gamma =\gamma '$, and $\gamma =\alpha'$, in which case $\alpha=\gamma =\alpha '=\gamma '$ and again the boards are identical. We conclude that the representatives $M$ are unique.
\end{proof}

Following Theorem \ref{T:Hmm_rep} we find that there are only nine possible $H_{mm}$-representatives, corresponding to the following pairs $[\alpha ,\gamma ]$:

\begin{center}
\begin{tabular}{lll}
$[1,1]$ & $[2,2]$ & $[7,7]$ \\
$[1,2]$ & $[2,1]$ & $[7,2]$ \\
$[1,8]$ & $[2,7]$ & $[7,5]$
\end{tabular}
\end{center}

\noindent
For example, the modular-magic Sudoku board shown in Figure~\ref{F:722} is the representative board $[7,2]$.

As mentioned in the proof of Lemma \ref{l:twoequal}, the set of modular-magic boards is a union of two $G_{mm}$-orbits. Observe that the three $H_{mm}$-nests represented by $[1,1]$, $[2,2]$, and $[7,7]$ lie in the
$G_{mm}$-orbit containing the left board of Figure \ref{F:representatives}, which has size $4608$ according to \cite{LorchWeld}. Meanwhile, the remaining six $H_{mm}$-nests lie in the same $G_{mm}$-orbit as the right-hand board of Figure \ref{F:representatives}, which has size $27648$ by \cite{LorchWeld}. This tells us that the three $H_{mm}$-nests represented by $[1,1]$, $[2,2]$, and $[7,7]$ have size $4608/3 =1536$ each while the remaining six $H_{mm}$-nests are each of size $27,648/6=4608$.

\subsection{A minimal complete modular-magic Sudoku symmetry group}
\label{SS:mm_minimal}

The modular-magic Sudoku relabeling symmetries group $S_{mm}$ described in Section \ref{SS:mm_sym} can be expressed as
$$S_{mm} = \langle \rho, \mu_{4,0}, \mu_{5,3}, \mu_{5,6} \rangle,$$

\noindent
since the four permutations $\rho = (12)(45)(78)$, $\mu_{4,0}(n)=(147)(285)$,
$\mu_{5,3}(n)=(03)(187245)$, and $\mu_{5,6}=(06)(127548)$ generate the entire group.

Now define {\em $H_{mm}$-nest graph for a group $S$} to be the graph that consists of nine vertices, one for each modular-magic $H_{mm}$-representative board, where two vertices $A$ and $B$ are connected by a directed edge $\sigma$ if the permutation $\sigma \in S$ takes the modular-magic representative board $A$ to a board that is $H_{mm}$-equivalent to representative board $B$.  It is sufficient to consider edges defined by a set of generators for $S$.  Since the set of modular-magic Sudoku boards has two orbits under the action of $G_{mm}=S_{mm} \times H_{mm}$ (see proof of Lemma \ref{l:twoequal}), the $H_{mm}$-nest graph for $S_{mm}$ corresponding to the four permutations $\rho$, $\mu_{4,0}$, $\mu_{5,3}$ and $\mu_{5,6}$ must have two components.

If $S'$ is a subgroup of $S_{mm}$, then $S' \times H_{mm}$ is a {\em complete modular-magic Sudoku symmetry group} if the $H_{mm}$-nest graph for $S'$ corresponding to a set of generators for $S'$ has two components.  In fact, if we take
$$S'=\langle \rho, \mu_{4,0} \rangle,$$

\noindent
then this is precisely what happens, as shown in Figure~\ref{F:mm_graph}.  In this figure the single arrow represents the permutation $\rho$ and the double arrow represents $\mu_{4,0}$.

\begin{figure}[h]
$$\xymatrix{[1,2]\ar@{<->}[r]\ar@2{->}[d]&[2,1]\ar@/^2pc/@{=>}[dd]&&[1,1]\ar@{<->}[d]\ar@/^1pc/@{=>}[d]\\
[2,7]\ar@{<->}[r]\ar@2{->}[d]&[1,8]\ar@2{->}[u]&&[2,2]\ar@{=>}[d]\\
[7,5]\ar@{<->}[r]\ar@/^2pc/@{=>}[uu]&[7,2]\ar@2{->}[u]&&[7,7]\ar@(dr,dl)@{<->}[]
}$$
\caption{Action of $\langle \rho, \mu_{4,0} \rangle$ on the set of $H_{mm}$-nests.}
\label{F:mm_graph}
\end{figure}
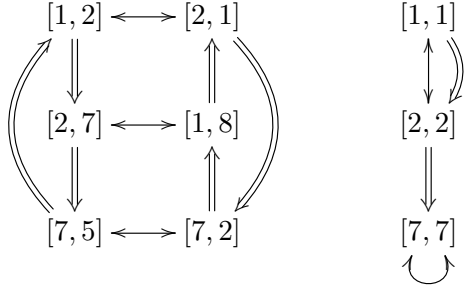

Furthermore, since $\langle \rho, \mu_{4,0} \rangle$ has order 27,648, which is equal to the largest orbit of $G_{mm}$, we know that this group is of {\em minimal} size.  This proves our first main result of this paper:

\begin{theorem} \label{T:maintheorem1}
$H_{mm} \times\langle \rho,\mu_{4,0}\rangle$ is a minimal complete modular-magic sudoku symmetry group.
\end{theorem}	

\section{Semi-magic Sudoku}
\label{S:sm}

A {\em semi-magic square} is a $3 \times 3$ array containing all of the symbols $\{0, 1, \dots, 8\}$ with each row and column adding to 12, and no condition on the diagonals.  A {\em semi-magic Sudoku board} is a Sudoku board whose $3 \times 3$ subsquares are semi-magic, see Figure \ref{MD1}.

\begin{figure}[h]
\begin{center}
\begin{sudoku-block}
|0|4|8|7|2|3|5|6|1|.
|5|6|1|0|4|8|7|2|3|.
|7|2|3|5|6|1|0|4|8|.
|8|0|4|1|5|6|3|7|2|.
|1|5|6|3|7|2|8|0|4|.
|3|7|2|8|0|4|1|5|6|.
|4|8|0|2|3|7|6|1|5|.
|6|1|5|4|8|0|2|3|7|.
|2|3|7|6|1|5|4|8|0|.
\end{sudoku-block}
\vspace{-1\baselineskip}
\end{center}
\caption{A semi-magic Sudoku board. }
\label{MD1}
\end{figure}

\subsection{Properties of semi-magic Sudoku}
\label{SS:properties}

\begin{lemma} \label{thelemma} Mini-rows in a semi-magic Sudoku subsquare must be permutations of $\{0, 4, 8\}, \{5, 6, 1\},$ and $\{7, 2, 3\}$  and the  mini-columns permutations of $\{0, 5, 7\},
\{4, 6, 2\}$, and $\{8, 1, 3\}$ or vice versa. (See Figure \ref{MD1} as an example.)
\end{lemma}

\begin{proof} A simple combinatorial argument shows that these are the only two ways to partition the set $\{0, \dots, 8\}$ into three sets of three that sum to twelve.  Once the first mini-row has been set, all other mini-rows must be from one subset and all mini-columns from the other.
\end{proof}

From the lemma, we can conclude that there are $3! \cdot 3! \cdot 2 = 72$ distinct $3 \times 3$ semi-magic Sudoku subsquares.  We will use the term {\em gnomon} to denote the union of the first pillar and first band of a semi-magic Sudoku board. Again, using the lemma, we see there are $72 \cdot 3! \cdot 2 \cdot 3! =72^2$ possible semi-magic Sudoku bands and $72^3$ semi-magic Sudoku gnomons.  We call the gnomon in Figure \ref{gnomon} the {\em standard gnomon}.

\begin{figure}[h]
\begin{center}
\begin{sudoku-block}
|0|4|8|7|2|3|5|6|1|.
|5|6|1|0|4|8|7|2|3|.
|7|2|3|5|6|1|0|4|8|.
|8|0|4||||||.
|1|5|6||||||.
|3|7|2||||b|||.
|4|8|0|||a||||.
|6|1|5|||||||.
|2|3|7|||||||.
\end{sudoku-block}
\vspace{-1\baselineskip}
\end{center}
\caption{The standard semi-magic Sudoku gnomon.}
\label{gnomon}
\end{figure}

The {\em gnomon-preserving} physical symmetries are generated by transpose, any row swap within a band or column swap within a pillar, and swapping pillars two and three or swapping bands two and three.  We denote the group generated by these symmetries $H_{\Gamma}$. Following the method used in Section~6 of \cite{Arnold2}, we partition the set of modular-magic Sudoku boards into {\em $H_{\Gamma}$-nests}, where two semi-magic Sudoku boards are in the same nest if and only if one can be obtained from the other by a sequence of physical symmetries from $H_{\Gamma}$.  The following theorem describes a unique representative for each $H_{\Gamma}$-nest:

\begin{theorem} Using {\em gnomon-preserving} physical symmetries from $H_{\Gamma}$, any semi-magic Sudoku board can be transformed so that its gnomon is the standard gnomon. There are sixteen $H_{\Gamma}$-nests, uniquely represented by a board of the form $[a,b]$ as shown in Figure \ref{gnomon}.
\end{theorem}

\begin{proof} We can easily take a semi-magic Sudoku board and set the standard gnomon using transpose, row, column,  and 2-3-band and pillar swaps.  Once the standard gnomon has been set, the board is completely determined by the entries in the (7,6) and (6,7) position in the $9 \times 9$ grid. The possible entries in the (7,6) position are $\{4, 5, 6, 0\}$ and $\{1,2,3,6\}$ in the (6,7) position. Therefore there are 16 distinct semi-magic Sudoku boards with the standard gnomon.  \end{proof}

We call these 16 representatives the {\em standard semi-magic Sudoku boards}, and denote them by $[a,b]$, as in Figure \ref{gnomon}.  For example, the semi-magic Sudoku board in Figure \ref{MD1} is denoted [7,1]. We have now determined that there are $72^3 \cdot 16 =$ 5,971,968 distinct semi-magic Sudoku boards.

All of the physical Sudoku symmetries from \cite{Arnold2} are valid semi-magic Sudoku symmetries, denoted $H_9$.  On the other hand, the group $S_{sm}$ of semi-magic Sudoku relabeling symmetries is far smaller than the group of sudoku relabelings.
One can show that $S_{sm}$ is isomorphic to the group of physical symmetries preserving semi-magic squares; meanwhile Lemma \ref{thelemma} indicates that this group of physical symmetries is isomorphic to $(S_3\times S_3) \rtimes \mathbb{Z}_2$, generated by row permutations, column permutation, and transpose. Therefore $S_{sm}\cong (S_3 \times S_3) \rtimes \mathbb{Z}_2$, and so the full group of semi-magic Sudoku symmetries, $G_{sm} = H_9 \times S_{sm}$, has order 3,359,232$ \cdot $72 = 241,864,704.  As with modular magic Sudoku and Shidoku, the size of this group is large compared with the set of semi-magic Sudoku boards that it is acting upon.  In the next section, we use the techniques of \cite{Arnold2} to find a minimal, complete group of symmetries for semi-magic Sudoku.

\subsection{Orbits and $H_{\Gamma}$-Nests for semi-magic Sudoku}
\label{S:sm_nests}

As described in the previous section, the sixteen boards denoted $[a,b]$ are representatives of the $H_{\Gamma}$-nests.   Clearly each nest sits inside a $G_{sm}$ orbit.  We need to determine which nests are in the same orbits.   Applying additional non-gnomon-preserving physical symmetries to these boards, we find the four orbits shown below.  In the diagram,
the single arrow is the symmetry, $u$,  swapping band 1 and 2 and the double arrow, $v$, is the symmetry swapping pillar 1 and 2.   Adding just a single relabeling, $\mu = (12)(45)(78)$, connects the middle two connected components in the diagram with the dashed line giving us three distinct semi-magic Sudoku components, denoted, top to bottom, $\mathcal{O}_1, \mathcal{O}_2$ and $\mathcal{O}_3$.

\begin{figure}[h] \label{orbits}
$$\hspace{-.5in}\xymatrix{&&&&[7,8]\ar@(dl,dr)@{=>}[]\ar@(ul,ur)@{->}[]&&&&\\
\\
&&[5,8]\ar@{=>}@(ul,ur)[]\ar[dl]\ar@{-->}[rrrr]&&&&[7,4]\ar@(ul,ur)[]\ar@{=>}[dl]&&\\
&[7,1]\ar@{=>}@(dl,dr)[]\ar[rr]&&[5,1]\ar@{=>}@(dl,dr)[]\ar[ul]&&[2,8]\ar@(dl,dr)[]\ar@{=>}[rr]&&[2,4]\ar@(dl,dr)[]\ar@{=>}[ul]\\
\\
&[5,4]\ar@{=>}@/^1pc/[rrr]\ar[dl]&&&[6,8]\ar@{=>}@/^1pc/[rrr]\ar[dl]&&&[2,6]\ar@{=>}@/_3pc/[llllll]\ar[dl]&\\
[7,6]\ar@{=>}@/_1pc/[rrr]\ar[rr]&&[6,1]\ar@{=>}@/_3pc/[rrrrrr]\ar[ul]&[2,1]\ar@{=>}@/_1pc/[rrr]\ar[rr]&&[5,6]\ar@{=>}@/_1pc/[lll]\ar[ul]&[6,4]\ar@{=>}@/^3pc/[llllll]\ar[rr]&&[6,6]\ar@{=>}@/_1pc/[lll]\ar[ul]
}$$
\vspace{.1in}
\caption{Action of $\langle u, v, \mu \rangle$ on $H_{\Gamma}$-nests}
\end{figure}
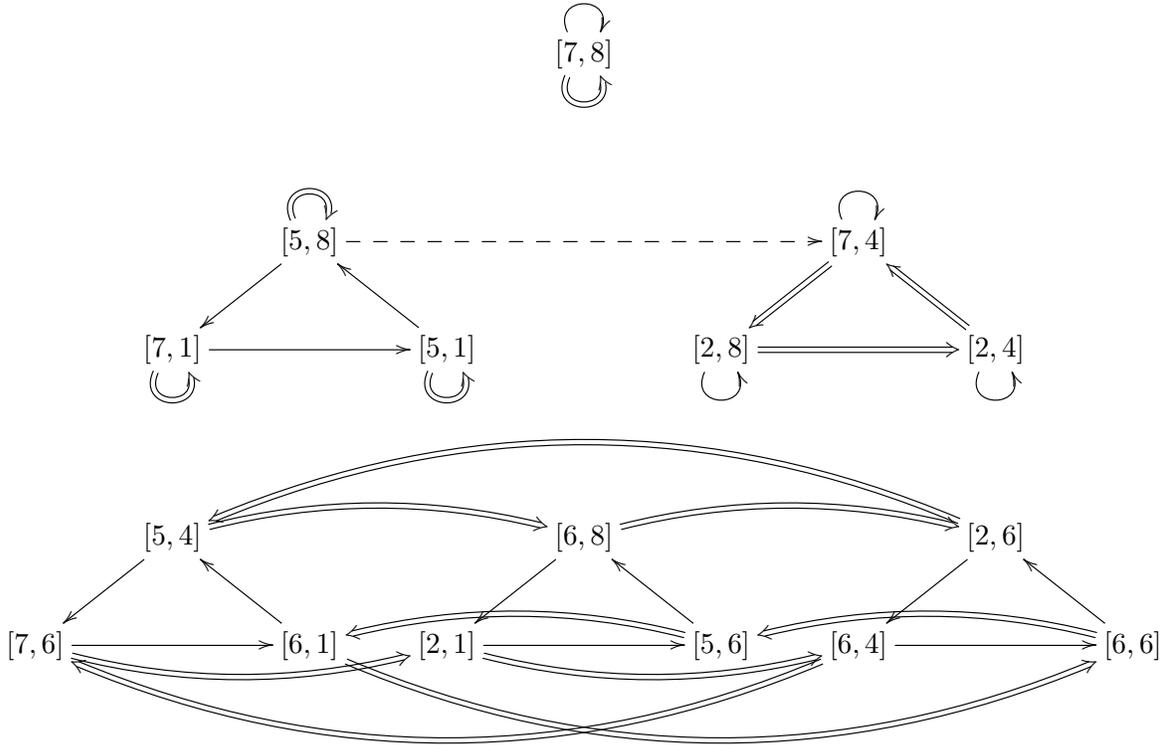

The diagram shows that there are at most three $G_{sm}$ orbits.  A brute force computation can confirm that there are no fewer.  However, a more elegant argument uses the notion of {\em Keedwell} boards and linearity from \cite{LorchJAuMS}.  In general, a Sudoku board is Keedwell if any subsquare can be obtained by permuting the rows and/or columns of the upper-left subsquare.  More precisely, we have the following definition from \cite{Keedwell}:

\begin{definition}  Let $\alpha$ be the operator on subsquares that acts by cycling down one mini-row and let $\beta$ the operator on subsquares that acts by cycling right one mini-column.  A Sudoku board $B$ with upper-left subsquare $K$ is {\bf Keedwell} if there exists matrices $\{c_{ij}\}$ and $\{d_{ij}\}$ such that
\begin{itemize}
\item $c_{00}=1$ and $d_{00}=1$, and
\item the $(i,j)^{\rm th}$ subsquare of $B$ is $\alpha^{c_{ij}}\beta^{d_{ij}}K$.
\end{itemize}
\end{definition}

For example, the Sudoku board $[7,1]$ shown in Figure~\ref{MD1} and in set $\mathcal{O}_2$ of Figure~\ref{orbits} is Keedwell; its upper-left block is
$K=\left[\begin{smallmatrix}
0 & 4 & 8 \\
5 & 6 & 1 \\
7 & 2 & 3
\end{smallmatrix}\right],$
and with respect to this block, board $[7,1]$ is of the form
$$[7,1]= \begin{array}{|r|r|r|}
\hline
K & \alpha K & \alpha^2 K \\
\hline
\beta K & \alpha^2 \beta K & \alpha \beta K \\
\hline
\beta ^2 K & \alpha \beta^2 K & \alpha^2 \beta^2 K \\
\hline
\end{array}\;.$$

Interestingly, all sixteen of the standard semi-magic Sudoku boards shown in Figure~\ref{orbits} are Keedwell.  For example, board $[7,8]$ from set $\mathcal{O}_1$ of Figure~\ref{orbits} and board $[7,6]$ from set $\mathcal{O}_3$ have the same upper-left block $K$ as $[7,1]$ and are of the form
$$[7,8]=\begin{array}{|r|r|r|}
\hline
K & \alpha K & \alpha ^2 K \\
\hline
\beta K & \alpha \beta K & \alpha^2 \beta K \\
\hline
\beta ^2 K & \alpha \beta ^2 K & \alpha^2 \beta^2 K \\
\hline
\end{array}
\quad \mbox{and} \quad
[7,6]= \begin{array}{|r|r|r|}
\hline
K & \alpha K & \alpha ^2 K \\
\hline
\beta K & \alpha^2 \beta K & \alpha \beta^2 K \\
\hline
\beta ^2 K & \alpha \beta^2 K & \alpha^2 \beta K \\
\hline
\end{array}\;.$$

We will show that each of the three collections of $H_{\Gamma}$-nests shown in Figure~\ref{orbits} can be distinguished by the following notion of linearity degree:

\begin{definition} \label{linear}
A matrix $\{m_{ij}\}$ is {\bf \em quasi-linear} if $m_{ij}=m_{i0}+m_{0j}$.  Suppose $B$ is a Keedwell Sudoku board with upper-left block $K$ and exponent matrices $\{c_{ij}\}$ and $\{d_{ij}\}$ for the cycles $\alpha$ and $\beta$.  Then the {\bf \em linearity degree} of $B$ is equal to the number of its exponent matrices that are quasi-linear.
\end{definition}

For example, board $[7,8]$ from orbit $\mathcal{O}_1$ has linearity degree 2, board $[7,1]$ from orbit $\mathcal{O}_2$ has linearity degree 1, and board $[7,6]$ from orbit $\mathcal{O}_3$ has linearity degree 0.  As we will soon see, the collections $\mathcal{O}_i$ shown in Figure~\ref{orbits} are in fact completely characterized by linearity degree, and this fact will enable us to prove that these three collections are in fact distinct orbits of $G_{sm}$.

Now let $G_k$ be the set of {\em Keedwell-preserving symmetries}; that is, the largest subgroup of the full Sudoku symmetry group whose elements preserve the set of Keedwell boards.  It is easy to see that all relabeling symmetries in $S_9$ are Keedwell-preserving, as well as compositions of transpose, pillar swaps, band swaps, 3-cycle permutations of rows within a band, and 3-cycle permutations of columns within a pillar.  The final group of Keedwell-preserving transformations is the set of {\em triple-transpositions} of rows (or columns) consisting of one row transposition in each of the three bands (or one column transposition in each of the three pillars).  Note that triple-transpositions reverse the ``orientation'' of all three bands (or pillars) in the sense that the order of the mini-rows (or mini-columns) of each band (or pillar) changes by a odd-degree permutation.  It is the fact that triple-transpositions reverse the orientation of {\em all} bands (or pillars) simultaneously that makes triple-transpositions Keedwell-preserving.  In fact, $G_k$ consists precisely of the symmetries that either preserve orientation in all the pillars/bands or reverse the orientation in all pillars/bands.  As a result of this, we have the following lemma.

\begin{lemma}\label{lemma1} Let $B_1$  be a Keedwell board and $g$ an element of the full Sudoku symmetry group.  If $g\cdot B_1 = B_2$ and $B_2$ is Keedwell, then $g \in G_k$.\end{lemma}

Now we can relate linearity degree to Keedwell-preserving symmetries.

\begin{lemma}\label{lemma2} Linearity degree is invariant under $G_k$. \end{lemma}
\begin{proof}
Clearly linearity degree is preserved by relabelings, transpose, pillar and band swaps, and 3-cycle permutations of rows within a band or columns within a pillar.   The only non-trivial case is to prove that triple-transpositions preserve linearity degree.  We will prove this case for pillars; the case for bands is similar.

Suppose we transform a Keedwell Sudoku board $B$ by a triple-transposition $g$ given by transpositions $\tau_1$ in the first pillar, $\tau_2$ in the second pillar, and $\tau_3$ in the third pillar.  If the original board $B$ is given by
$$B=\begin{array}{|r|r|r|}
\hline
K & \alpha^{c_{01}}\beta^{d_{01}} K & \alpha^{c_{02}}\beta^{d_{02}} K \\
\hline
\alpha^{c_{10}}\beta^{d_{10}} K & \alpha^{c_{11}}\beta^{d_{11}} K & \alpha^{c_{12}}\beta^{d_{12}} K \\
\hline
\alpha^{c_{20}}\beta^{d_{20}} K & \alpha^{c_{21}}\beta^{d_{21}} K & \alpha^{c_{22}}\beta^{d_{22}} K \\
\hline
\end{array} \;,$$

\noindent
then $gB$ is the Keedwell Sudoku board given by
$$gB=\begin{array}{|r|r|r|}
\hline
\tau_1 K & \tau_2 \alpha^{c_{01}}\beta^{d_{01}} K & \tau_3 \alpha^{c_{02}}\beta^{d_{02}} K \\
\hline
\tau_1 \alpha^{c_{10}}\beta^{d_{10}} K & \tau_2 \alpha^{c_{11}}\beta^{d_{11}} K & \tau_3 \alpha^{c_{12}}\beta^{d_{12}} K \\
\hline
\tau_1 \alpha^{c_{20}}\beta^{d_{20}} K & \tau_2 \alpha^{c_{21}}\beta^{d_{21}} K & \tau_3 \alpha^{c_{22}}\beta^{d_{22}} K \\
\hline
\end{array} \;.$$

\noindent
For each $k$ we have $\tau_k \alpha = \alpha \tau_k$ and $\tau_k \beta = \beta^2 \tau_k$, so for all $i,j,k$ we have
$$\tau_k \alpha^{c_{ij}} \beta^{c_{ij}} K = \alpha^{c_{ij}} \tau_k \tau_1 \beta^{2d_{ij}} (\tau_1 K).$$

\noindent
Since $\tau_k \tau_1$ is a 3-cycle for each $k$, we have $\tau_2 \tau_1 \beta^{2d_{i1}} = \beta^{2d_{i1}+r}$ and $\tau_3 \tau_1 \beta^{2d_{i2}} = \beta^{2d_{i2}+s}$.  Therefore $gB$ can be written
$$gB=\begin{array}{|r|r|r|}
\hline
\tau_1 K & \alpha^{c_{01}}\beta^{2d_{01}+r} (\tau_1 K) & \alpha^{c_{02}}\beta^{2d_{02}+s} (\tau_1 K) \\
\hline
\alpha^{c_{10}}\beta^{2d_{10}} (\tau_1 K) & \alpha^{c_{11}}\beta^{2d_{11}+r} (\tau_1 K) & \alpha^{c_{12}}\beta^{2d_{12}+s} (\tau_1 K) \\
\hline
\alpha^{c_{20}}\beta^{2d_{20}} (\tau_1 K) & \alpha^{c_{21}}\beta^{2d_{21}+r} (\tau_1 K) & \alpha^{c_{22}}\beta^{2d_{22}+s} (\tau_1 K) \\
\hline
\end{array} \;,$$

\noindent
which clearly has the same linearity degree as $B$.
\end{proof}

With the two previous lemmas we are now able to show that the three connected components $\mathcal{O}_1$, $\mathcal{O}_2$, and $\mathcal{O}_3$ of $H_{\Gamma}$-nests from Figure~\ref{orbits} are in fact precisely the orbits of the semi-magic Sudoku boards under the action of $G_{sm}$.

\begin{theorem} There are exactly three $G_{sm}$-orbits on the set of semi-magic Sudoku boards.
\end{theorem}
\begin{proof}  We have already produced three sets of semi-magic Sudoku boards, $\mathcal{O}_1, \mathcal{O}_2$ and $\mathcal{O}_3$ shown in Figure \ref{orbits}, that are connected by elements of $G_{sm}$.  Suppose, for example, that $\mathcal{O}_1$ and $\mathcal{O}_2$ were not distinct $G_{sm}$-orbits.  Then there exists a $g \in G_{sm}$ such that $g \cdot [7,8] = [5,8]$.  By Lemma \ref{lemma1}, $g \in G_k$.  But Lemma \ref{lemma2} states that [7,8] and [5,8] have the same linearity degree.  This is a contradiction to the fact that [7,8] has linearity degree 2 and [5,8] has linearity degree 1.  Therefore, there are exactly three $G_{sm}$-orbits of semi-magic Sudoku boards
\end{proof}

\subsection{A minimal complete semi-magic Sudoku symmetry group}
\label{S:sm_minimal}

Since each standard semi-magic Sudoku board represents $72^3$ distinct semi-magic Sudoku boards, the three orbits described in Section \ref{S:sm_nests} have order $72^3$, $6 \cdot 72^3$ and $9 \cdot 72^3$.  The order of each orbit must divide the order of any semi-magic Sudoku symmetry group.  Therefore, a minimal semi-magic Sudoku symmetry group must be a multiple of lcm$(72^3, 6 \cdot 72^3, 9 \cdot 72^3) = 18 \cdot 72^3$.  The group used in producing the three orbits in Figure \ref{orbits} consists of all of the Sudoku physical symmetries and the relabeling symmetry, (12)(45)(78).  This group, $G = H_9 \times \langle (12)(45)(78) \rangle$, has order $18 \cdot 72^3$, so is, in fact, a minimal complete semi-magic Sudoku symmetry group.

\section{A minimal complete Sudoku symmetry group}

A natural question to ask is whether the techniques we used in this paper to investigate modular-magic Sudoku and semi-magic Sudoku can be applied to standard $9\times 9$ Sudoku to reduce the size of the Sudoku symmetry group.  The full physical Sudoku symmetry group $H_9$ contains all possible band, pillar, row and column swaps as well as all of the symmetries of the square.  This group has order 3,359,232 and as all $9!$ elements of $S_9$ are valid relabelings, the full Sudoku symmetry group $G_9=H_9\times S_9$ has order 1,218,998,108,160.  In fact, this group is {\em already} minimal because there exist Sudoku boards that are not fixed by any non-identity element of $G_9$ so the size of the largest orbit is $|G_9|$.

To see this, consider that there are 6,670,903,752,021,072,936,960 possible Sudoku boards and $N=$5,472,730,538 orbits  under the action of $G_9$ \cite{Felgenhauer}. Therefore the average size of an orbit is 1,218,935,174,261. Suppose, for a contradiction, that every Sudoku board is fixed by at least one non-identity element of $G_9$. If the  $N$ orbits have corresponding stabilizer groups $K _1,\dots ,K _N$ then
$$
\text{Average orbit size} =\frac{\frac{|G_9|}{|K_1|} +\cdots +\frac{|G_9|}{|K_N|}}{N}
\leq \frac{\frac{|G_9|}{2} +\cdots +\frac{|G_9|}{2}}{N} = \frac{1}{2} |G_9|,
   $$
which is clearly far less than the actual average orbit size stated above. Therefore at least one Sudoku board is not fixed by any non-identity element of $G_9$.

Since the full Sudoku symmetry group is already minimal, the techniques in this paper cannot be used to reduce it.  However, these techniques should be helpful in analyzing other types of puzzles, including Sudoku variants.  As seen in \cite{Arnold2}, reduction of the symmetry group can be of great practical use towards the goal of analyzing Sudoku-style puzzles  from a theoretical perspective.


\end{document}